\documentclass[11pt]{article}
\usepackage{glaudo}

\title{The Brauer group of tori}

\author{Julian Demeio}

\begin{document}
	
	\maketitle
	
	\begin{abstract}
		\noindent We show that the map $\Br T \to (\Br  T_{\ok})^{\Gamma_k}$ is surjective for a torus $T$ defined over a field $k$ of characteristic $0$ when $k$ is a local or global field or $T$ is quasi-trivial.
	\end{abstract}
	
	\section{Introduction}\label{Sec1}

        The {\em Brauer group} of a scheme $X$ is the second étale cohomology group $H^2_{\et}(X,\G_m)$. When $X$ is a variety over a perfect field $k$, we have a natural map
        \[
        \Br X \to (\Br X_{\ok})^{\Gamma_k}
        \]
        from Brauer classes of $X$ to Galois-invariant Brauer classes of $X_{\ok}$. Its kernel $\Br_1X$ is the {\em algebraic Brauer group} of $X$, and its image $\Br_{tr}X$ is the {\em transcendental Brauer group} of $X$.
        
        In this note, we are interested in the Brauer group of a torus $T$ over a field $k$ of characteristic $0$. The algebraic Brauer group of $T$ is well understood (see e.g.\ Section 9.1 of \cite{BGbook}), so we focus our attention on the transcendental Brauer group. It is known that the exponent of the quotient $ (\Br X_{\ok})^{\Gamma_k} / \Br_{tr}X$ divides $2$, see e.g.\ the proof of \cite[Proposition 9.1.3]{BGbook} in Colliot-Thélène and Skorobogatov's book.
	Our main theorem is:
	\begin{theorem}\label{Thm1}
            Let $k$ be a field of characteristic $0$, and $T$ be a $k$-torus. The identity $ \Br_{tr}X = (\Br X_{\ok})^{\Gamma_k} $ holds if either:
		\begin{enumerate}[label=(\roman*)]
			\item the torus $T/k$ is quasi-trivial;
			\item $k$ is a local or global field.
		\end{enumerate}
	\end{theorem}
	
	In case {(i)}, we prove in fact a more precise result (see Theorem \ref{Thm2}), that provides an explicit basis of the finite abelian group $\Br T/\Br_1T$. The exact expression for the elements of this basis was inspired by some Hilbert symbols appearing in the work  \cite{Shafarevich} (see also \cite[Theorem IX.9.3.2]{GermanBook}) of Shafarevich.
	Case {(ii)} is proven by analyzing the exact sequence
	\begin{equation}\label{Eq15}
		0 \to \Br T/\Br_1T \to (\Br T_{\ok})^{\Gamma_k} \to[d^{0,2}_3] H^3(k,\ok[T]^*),
	\end{equation}
	where $d^{0,2}_3$ is the third-page Hochschild-Serre spectral sequence differential. Since the cohomological dimension of non-archimedean local fields is $2$, and  when $k$ is a number field  $H^3(k,\ok[T]^*)= \bigoplus_{\substack{v \in M_K\\v \text{ real}}}H^3(k_v,\ok_v[T]^*)$ by the Poitou-Tate theorems, the only interesting case is that of $k=\R$. 
	In this case, we use the theory of characteristic classes for the 
	Hochschild-Serre spectral sequence of group extensions developed by Charlap and Vasquez \cite{CV1}\cite{CV2} (and a subsequent result by Sah \cite{Sah1}) to show that the differential $d^{0,2}_3$ is trivial.
	
	\begin{remark*}
		We ignore whether the injection $\Br T/\Br_1T \to (\Br T_{\ok})^{\Gamma_k}$ is an isomorphism for a general torus $T$ defined over an arbitrary field $k$ of characteristic $0$, with a slight suspicion that this should not be the case as there seems to be no transparent reason for the differential $d^{2,0}_3$ to be trivial (our proof of its triviality in the real case heavily relies on the cyclicity of the absolute Galois group of $\R$). 
	\end{remark*}

        \vskip3mm

        {\bf Structure of the paper.} In Section \ref{Sec2} we fix some notation. In Section \ref{Sec3} we state and prove Theorem \ref{Thm2}. In Section \ref{Sec4} we state and prove Theorem \ref{Thm3} (i.e.\ Theorem \ref{Thm1} for real tori), and prove Theorem \ref{Thm1}. In the appendix, we show a compatibility result between different Hochschild-Serre spectral sequences (this is used in the proof of Theorem \ref{Thm3}).
	
	\section{Notation}\label{Sec2}
	
	Let $k$ be a field of characteristic $0$, $\bar k$ be its algebraic closure and $\Gamma_k$ be the absolute Galois group $\Gal(\bar k/k)$. 
		
	Let $n$ be a natural number such that $\mu_n \subset k^*$, and $X$ be a $k$-variety. We denote by
	\begin{equation}\label{Eq:Hilbert}
	(-,-)_n : (k(X)^*/n)^{\otimes 2} \to[\cup] H^2(k(X), \mu_n^{\otimes 2})=H^2(k(X), \mu_n)(1) \to (\Br k(X))_n(1)
	\end{equation}
	the {\em Hilbert symbol} defined by composing the Hilbert 90 isomorphism $k(X)^*/n \to[\sim] \linebreak H^1(k(X), \mu_n)$ with the cup-product. Choosing a primitive $n$-th root of unity $\zeta$ gives an identification $(\Br k(X))_n(1)=(\Br k(X))_n$. For $f,g \in k(X)^*$, we let $(f,g)_{\zeta,n}$ be the image of $(f,g)_n$ under this identification. When the choice of $\zeta$ is clear from context, we drop it from notation.
	
	\section{Transcendental Brauer groups of quasi-trivial tori}\label{Sec3}
	Let $Q$ be a quasi-trivial torus over $k$ of rank $r$, and 
	\[
	y_1,\ldots,y_r:Q_{\ok} \to \G_{m,\bar k}
	\]
	be a $\Gamma_k$-equivariant basis of characters.
	For $1 \leq i < j \leq r$, we let:
	\begin{enumerate}
		\item $E_{ij}$ be the subfield of $\bar k$ such that  $\Gamma_{E_{ij}} < \Gamma_k$ is the stabilizer of the ordered pair $(y_i,y_j)$ under the $\Gamma_k$-action;
		\item $E'_{ij}$ be the subfield of $E_{ij}$ such that  $\Gamma_{E'_{ij}} < \Gamma_k$ is the stabilizer of the unordered pair $\{y_i,y_j\}$ under the $\Gamma_k$-action.
	\end{enumerate}
	Note that $E'_{ij} \subset E_{ij}$ and the extension $E_{ij}/E'_{ij}$ is either trivial or quadratic for every pair $i<j$. We also let:
	\begin{enumerate}[resume]
		\item $n_{ij}$ be the largest order of a root of unity contained in $E_{ij}$;
		\item $\sigma_{ij}$ be the generator of the group $\Gal(E_{ij}/E'_{ij})$ (this is non-trivial if and only if $E'_{ij}\neq E_{ij}$);
		\item $n'_{ij}\coloneqq \gcd(n_{ij},1+\chi(\sigma_{ij}))$, where $\chi$ denotes the cyclotomic character.
	\end{enumerate}
	
	We let $O$ be the set of orbits for the $\Gamma_k$-action on the unordered pairs $\{y_i,y_j\}, 1 \leq i < j \leq r$. For each orbit $o \in O$, let $i(o)< j(o)$ be a pair of indices such that $\{y_{i(o)},y_{j(o)}\} \in o$. Let $R\subset \binom{\{1,\ldots,r\}}{2}$ be the collection of pairs $\{i(o),j(o)\}_{o\in O}$.
	
	We fix an inverse system of roots of unity $(\zeta_n)_{n \in \N} \subset \bar k^*, \zeta_m^{m/n}=\zeta_n$ whenever $n \mid m$.
	\begin{theorem}\label{Thm2}        
		The following elements of $\Br k(Q)$ lie in $\Br Q \subset \Br k(Q)$:
		\begin{enumerate}[label={\em(\Roman*)}]
			\item $\cores_{E_{ij}/k}(y_i,y_j)_{n_{ij}}$ for every pair $i<j$,
			\item $\cores_{E_{ij}/k}(y_i,y_j-y_i)_{n'_{ij}}$ for every pair $i<j$ such that $E_{ij}/E'_{ij}$ is quadratic.
		\end{enumerate}
            Moreover, the images of (I) and (II) under the natural map
		\[
		\Br Q/\Br_1 Q \to (\Br Q_{\ok})^{\Gamma_k}
		\]
		generate $(\Br Q_{\ok})^{\Gamma_k}$ and they form a basis of this abelian group when we restrict only to pairs $(i,j)$ with $\{i,j\} \in R$.
	\end{theorem}

	Here the symbols $(-,-)_{m} \coloneqq (-,-)_{m,\zeta_m}$ denote the Hilbert symbols
	\[
	(E_{ij}(Q)^*/m)^{\otimes 2} \to (\Br E_{ij}(Q))_m(1)=(\Br E_{ij}(Q))_m
	\]
	where $m=n_{ij}$ or $n'_{ij}$, and the symbols $\cores_{E_{ij}/k}$ denote the corestriction maps \linebreak $\Br E_{ij}(Q) \to \Br k(Q)$ along the finite field extensions $E_{ij}(Q)/k(Q)$.
	
	\begin{remark}\label{Rmk2}
		Since $y_i$ and $y_j$ define invertible functions on the scheme $Q \times_k E_{ij}$, that $\cores_{E_{ij}/k}(y_i,y_j)_{n_{ij}}$ lies in $\Br Q$  follows immediately from general corestriction formalism for schemes (see \cite[Sec.\ 3.8]{BGbook}), so the only non-trivial part of the first statement of Theorem \ref{Thm2} is for elements (II).
	\end{remark}
	
	\subsection{Proof of Theorem \ref{Thm2}}\label{SecRemove}
	\begin{lemma}\label{Lem1}
		Let $i<j$ be such that $E_{ij}/E'_{ij}$ is quadratic, then the element  $\cores_{E_{ij}/k}(y_i,y_j-y_i)_{n'_{ij}} \in \Br k(Q)$ lies in the subgroup $\Br Q \subset \Br k(Q)$.
	\end{lemma}
	
	\begin{proof}
		To lighten the notation, let $E \coloneqq E_{ij}, E' \coloneqq E'_{ij}, n' \coloneqq n'_{ij}, \sigma \coloneqq \sigma_{ij}, y\coloneqq y_i, y'\coloneqq y_j$. We have $\cores_{E/k}(y,y'-y)_{n'}  = \cores_{E'/k}\cores_{E/E'}(y,y'-y)_{n'}$. By general corestriction formalism for schemes (similarly as in Remark \ref{Rmk2}) it suffices to prove that $\cores_{E/E'}(y,y'-y)_{n'} \in \Br (Q \times_k E')$.
		
		Let $D \subset Q' \coloneqq Q \times_k E'$ be the (geometrically) irreducible divisor defined by the $\Gal(E/E')$-invariant equation $y=y'$, and let $U \coloneqq Q \s D$. Again by corestriction formalism, it is clear that $\cores_{E/E'}(y,y'-y)_{n'}\in \Br U$. To prove that $b' \coloneqq \cores_{E/E'}(y,y'-y)_{n'}$ lies in $\Br Q,$ it suffices to show that the residue at $D$ vanishes (see purity of the Brauer group \cite[Theorem 3.7.2]{BGbook}). Combining \cite[Proposition 1.4.7]{BGbook} and \cite[Proposition 1.5.9]{GermanBook}, we get a commutative diagram:
		\begin{equation}\label{Eq14}
		\begin{tikzcd}
		\Br U_E \arrow[r, "r"] \arrow[d, "\cores_{E/E'}"] & {H^1(E(D),\qz)=\Hom(\Gamma_{E(D)}^{ab},\qz)} \arrow[d, "V^*"] \\
		\Br U \arrow[r, "r"]                              & {H^1(E'(D),\qz)=\Hom(\Gamma_{E'(D)}^{ab},\qz)},       
		\end{tikzcd}
		\end{equation}
		where the horizontal maps are residue maps, and the second vertical map is pullback along the Verlagerung (transfer) map  $V: \Gamma_{E'(D)}^{ab} \to \Gamma_{E(D)}^{ab}$.    	
		Since $y'-y$ is a uniformizer for $D_E \coloneqq D \times_{E'} E$, the residue $r_b\coloneqq r(b)$ of $b\coloneqq \left(y, y'-y\right)_{n'}$ at $D_E$ is 
		the Kummer character
		\begin{gather*}
		r_b:\Gamma_{E(D)} \to \Z/n'\Z, \  \text{defined by} \ \gamma(\sqrt[n']y)=\zeta_{n'}^{r_b(\gamma)}\sqrt[n']y.
		\end{gather*}
		Using the defining formulas for the Verlagerung map (see e.g.\ \cite[p.52]{GermanBook}), we deduce from the commutativity \eqref{Eq14} that the residue $r_{b'}\coloneqq r(b')$ of $b'$ at $D$ is 
		\begin{equation}\label{Eqpsi1}
		r_{b'}:\Gamma_{E'(D)} \to \Z/n'\Z,\ \  \begin{cases}
		r_{b'}(\gamma) = r_{b}(\gamma) + r_{b}(\bar\sigma^{-1} \gamma \bar\sigma), \\
		r_{b'}(\bar\sigma \cdot \gamma ) = r_{b}(\gamma) + r_{b}(\bar\sigma \gamma \bar\sigma) = r_{b}(\gamma) + r_{b}(\bar \sigma^2)+ r_{b}(\bar\sigma^{-1} \gamma \bar\sigma) ,
		\end{cases}
		\end{equation}
		where $\bar\sigma \in \Gamma_{E'(D)}$ is any lift of $\sigma \in \Gal(E/E')$ to $\Gamma_{E'(D)}$ and $\gamma$ lies in the subgroup $\Gamma_{E(D)} < \Gamma_{E'(D)}$.
		Note that $\bar \sigma^2 \in \Ker r_{b}$. In fact, $\bar \sigma$ projects to a reflection in the dihedral Galois group 
		\begin{multline*}
		\Gamma_{E'(D)}/\Ker r_b=\Gamma_{E'(D)}/\Gamma_{E(D)(\sqrt[n']y)}=\Gal\left(E(D)(\sqrt[n']y)/E'(D) \right)=\\ \Gal\left(E(D)(\sqrt[n']y)/E(D) \right) \rtimes \Gal (E/E') \cong\Z/n'\Z(1) \rtimes \Gal (E/E') \cong \Z/n'\Z \rtimes \Z/2\Z,
		\end{multline*}
            and all reflections of dihedral groups have order $2$.
		Hence the expression \eqref{Eqpsi1} reduces (in both cases) to $r_b(\gamma) + r_b(\bar\sigma^{-1} \gamma \bar\sigma)=r_b(\gamma) + \chi(\sigma)^{- 1}r_b(\gamma)$. By definition of $n'$ we have
		$
		\chi(\sigma)^{-1}r_b(\gamma)=-r_b(\gamma),
		$
		and we obtain the sought vanishing.
	\end{proof}

	For $i \in \{1,\ldots,r\}$ and $\gamma \in \Gamma_k$, we define $\gamma(i) \in \{1,\ldots,r\}$ through the identity $y_{\gamma(i)} \coloneqq \gamma(y_i)$.
	The following lemma is technically not necessary to prove Theorem \ref{Thm2}, but it shows that the second statement of this theorem is independent of the choice of $R$.
	
	\begin{lemma}
			For every $\gamma \in \Gamma_k$:
			\begin{enumerate}[label={\em(\alph*)}]
				\item  $\cores_{E_{\gamma(i)\gamma(j)}/k}(y_{\gamma(i)},y_{\gamma(j)})_{n_{\gamma(i)\gamma(j)}}=\chi(\gamma)^{-1}\cores_{E_{ij}/k}(y_i,y_j)_{n_{ij}}$ for every pair $i<j$,
				\item $\cores_{E_{\gamma(i)\gamma(j)}/k}(y_{\gamma(i)},y_{\gamma(j)}-y_{\gamma(i)})_{n_{\gamma(i)\gamma(j)}}=\chi(\gamma)^{-1}\cores_{E_{ij}/k}(y_i,y_j-y_i)_{n'_{ij}}$ for every pair $i<j$ such that $E_{ij}/E'_{ij}$ is quadratic.
			\end{enumerate}
	\end{lemma}
	\begin{proof}
		To prove (a), consider the following commutative diagram:
		\begin{equation}\label{Eq2}
		\begin{tikzcd}
		Q \times_k E_{ij} \arrow[rr, "{(id,\gamma)}"] \arrow[rd, "pr_1"] &   & Q \times_k E_{\gamma(i)\gamma(j)} \arrow[ld, "pr_1"'] \\
		& Q &
		\end{tikzcd}
		\end{equation}
		where the vertical morphisms are projections and the horizontal isomorphism is induced by conjugation by $\gamma$ on $E_{ij}$ (note that $\gamma(E_{ij})=E_{\gamma(i)\gamma(j)}$). We have 
		\begin{equation}\label{Eq4}
		(id,\gamma)_*(y_i,y_j)_{n_{ij}}=\chi(\gamma)^{-1}(y_{\gamma(i)},y_{\gamma(j)})_{\zeta,n_{\gamma(i)\gamma(j)}},
		\end{equation}
		as follows from the sequence of identities:
		$
		(id,\gamma)_*(y_i,y_j)_{n_{ij}}=(id,\gamma)_*(y_i,y_j)_{\zeta,n_{ij}}=(y_{\gamma(i)},y_{\gamma(j)})_{\gamma(\zeta),n_{ij}}=\chi(\gamma)^{-1}(y_{\gamma(i)},y_{\gamma(j)})_{\zeta,n_{\gamma(i)\gamma(j)}},
		$
		where $\zeta\coloneqq \zeta_{n_{ij}}$.
		By the commutativity of \eqref{Eq2}, we have $\cores_{E_{ij}/k}=\cores_{E_{\gamma(i)\gamma(j)}/k}\circ (id,\gamma)_*:\Br(Q \times_k E_{ij}) \to \Br Q$, and (a) follows from \eqref{Eq4}. Point (b) is proven analogously.
	\end{proof}
	
	\begin{proof}[Proof of Theorem \ref{Thm2}]
		The first part follows from Lemma \ref{Lem1} and Remark \ref{Rmk2}. 
		We prove the second part.	Let $m$ be a common multiple of all the $n_{ij}$'s. By \cite[Proposition 9.1.2(ii)]{BGbook}, we have 
		\begin{multline}\label{Eq3}
		(\Br \overline{Q})[m]=\bigoplus_{1\leq i < j \leq r} \Z/m\Z \cdot (y_i,y_j)_m=\bigoplus_{o \in O}(\Br \overline{Q})[m]_o,\\ (\Br \overline{Q})[m]_o \coloneqq \bigoplus_{\{i,j\} \in o} \Z/m\Z \cdot (y_{i},y_{j})_m, \ o \in O,
		\end{multline}
		where $\Gamma_k$ acts via $\gamma((y_i,y_j)_m)=\chi(\gamma)^{-1}(y_{\gamma(i)},y_{\gamma(j)})_m$. We emphasize that each summand $(\Br \overline{Q})[m]_o$ is $\Gamma_k$-invariant. 
		
		Fix $o \in O$, and let $i\coloneqq i(o),j\coloneqq j(o)$. Taking Galois invariants in \eqref{Eq3} shows that $(\Br \overline{Q})[m]_o^{\Gamma_k}$ is cyclic generated by the element 
		\begin{equation}\label{El1}
		\tag*{(\thesection.\arabic{equation})$_o$} \addtocounter{equation}{1}
		\sum_{\gamma \in \Gamma_K/\Gamma_{E_{ij}}} \chi(\gamma)^{-1} (y_{\gamma (i)},y_{\gamma (j)})_{n_{ij}}, \ \text{ if } E_{ij}=E'_{ij},
		\end{equation}
		and
		\begin{equation}\label{El2}
		\tag*{(\thesection.\arabic{equation})$_o$} \addtocounter{equation}{1}
		\sum_{\gamma \in \Gamma_K/\Gamma_{E'_{ij}}} \chi(\gamma)^{-1} (y_{\gamma (i)},y_{\gamma (j)})_{n'_{ij}}, \ \text{ if } E_{ij}/E'_{ij} \text{ is quadratic}.
		\end{equation}    
		Letting $m$ go multiplicatively to $\infty$, we infer that $(\Br \overline{Q})^{\Gamma_k}$ is a finite abelian group with a basis given by \ref{El1} and \ref{El2} as $o$ varies in $O$.
		
		The restriction of $\cores_{E_{ij}/k}(y_i,y_j)_{n_{ij}} \in \Br Q$ to $\Br Q_{\ok}$ is equal to \cite[Lemma 5.4.13]{BGbook}:
		\[
		\sum_{\gamma \in \Gamma_K/\Gamma_{E_{ij}}} \gamma( (y_{i},y_{j})_{n_{ij}})=\sum_{\gamma \in \Gamma_K/\Gamma_{E_{ij}}} \chi(\gamma)^{-1} (y_{\gamma (i)},y_{\gamma (j)})_{n_{ij}}.
		\]
		When $E_{ij}=E'_{ij}$, this is exactly \ref{El1}. Let now $i<j$ be such that $E_{ij}/E'_{ij}$ is quadratic. Then
		the restriction of $\operatorname{cores}_{E_{i j} / k}\left(y_i, y_j-y_i\right)_{n_{i j}^{\prime}}$ to $\Br Q_{\ok}$ is equal to:
		\begin{equation}\label{Eq3.5}
		\sum_{\gamma \in \Gamma_K/\Gamma_{E_{ij}}} \gamma( (y_{i},y_{j}-y_i)_{n'_{ij}}).
		\end{equation}
		Note that 
		\begin{align*}
		\sum_{\sigma \in \Gamma_{E'_{ij}}/\Gamma_{E_{ij}}} \sigma( (y_{i},y_{j}-y_i)_{n'_{ij}})
		&= (y_{i},y_{j}-y_i)_{n'_{ij}} + \chi(\sigma) (y_{j},y_{i}-y_j)_{n'_{ij}} \\ 
		&= (y_{i},y_{j}-y_i)_{n'_{ij}} - (y_{j},y_{i}-y_j)_{n'_{ij}} \\
		&= (y_{i},y_{j}-y_i)_{n'_{ij}} - (y_{j},y_{j}-y_i)_{n'_{ij}} \\
		&= (y_{i}/y_j,y_{j}-y_i)_{n'_{ij}} \\
		&= (y_{i}/y_j,y_{j}(1-y_i/y_j))_{n'_{ij}} \\
		&= (y_{i}/y_j,y_{j})_{n'_{ij}} =  (y_{i},y_{j})_{n'_{ij}},
		\end{align*}        
		where we used the definition of $n'_{ij}$ in the second identity, and the Steinberg relations $(a,1-a)=(a,-a)=(a,a)=0$ (the last holds because $-1 \in \ok^*$) in the rest. Thus \eqref{Eq3.5} becomes:
		\begin{multline*}
		\sum_{\gamma \in \Gamma_K/\Gamma_{E'_{ij}}} \sum_{\sigma \in \Gamma_{E'_{ij}}/\Gamma_{E_{ij}}}\gamma\sigma( (y_{i},y_{j}-y_i)_{n'_{ij}}) =\sum_{\gamma \in \Gamma_K/\Gamma_{E'_{ij}}}  \gamma( (y_{i},y_{j})_{n'_{ij}})\\ =\sum_{\gamma \in \Gamma_K/\Gamma_{E'_{ij}}} \chi(\gamma)^{-1} (y_{\gamma (i)},y_{\gamma (j)})_{n'_{ij}}.
		\end{multline*}
		So the restriction of the families (I) and (II) to $\Br Q_{\ok}$ gives exactly the basis of $(\Br Q_{\ok})^{\Gamma_k}$ formed by \ref{El1} and \ref{El2}.
	\end{proof}

	\section{Transcendental Brauer groups of real tori}\label{Sec4}

	\begin{theorem}\label{Thm3}
		Let $T$ be an algebraic torus over the real field $\R$. Then the natural map
		\[
		\Br T \to (\Br T_{\C})^{\Gal(\C/\R)}
		\]
		is surjective.
	\end{theorem}
	
	To prove this, we use the theory of Charlap--Vasquez \cite{CV1}\cite{CV2}. Let $G$ be a semi-direct product $N \rtimes \pi$, where $N$ is a free finitely generated abelian group, and $\pi$ acts by automorphisms. Then Charlap and Vasquez construct {\em characteristic classes}
	\[
	v_2^n(N) \in H^2(\pi, H^{n-1}(N, H_n(N,\Z)))
	\]
	and relate them to the Hochschild-Serre differentials (see {\em loc.cit.}).
	We only need a relatively simple case of their theory (Theorem \ref{ThmCV} below), which may be proven without the lengthy computations of \cite{CV1}\cite{CV2}. 
	
	Since $N$ is $\Z$-free, the universal coefficient theorem gives an identification of $\pi$-modules
	\[
	H^2(N,M) = \Hom(H_2(N),M), \ \  H_2(N) \coloneqq H_2(N,\Z),
	\]
	and thus of groups
	\begin{equation}\label{Eq6}
	H^2(N,M)^{\pi} = \Hom_{\pi}(H_2(N),M).
	\end{equation}
	
	
	\begin{theorem}[Charlap-Vasquez]\label{ThmCV}
		There exists a class
		\[
		v_2 =v_2(N) \in H^2(\pi, H^1(N, H_2(N)))
		\]
		such that, for any $\pi$-module $M$, the Hochschild-Serre differential 
		\[
		d_2^{0,2}:H^2(N,M)^{\pi} \to H^2(\pi,H^1(N,M))
		\]
		sends $\alpha$ to $\tilde \alpha_*v_2$, where $\tilde \alpha \in \Hom_{\pi}(H_2(N),M)$ is the element corresponding to $\alpha$ under the identification \eqref{Eq6}.
	\end{theorem}
	\begin{proof}
		This is just a consequence of the naturality of the Hochschild-Serre spectral sequence. Namely, 
		let $M_{univ}$ be the $\pi$-module $H_2(N, \Z)$, and $\id \in H^2(N,M_{univ})=\Hom_\pi(H_2(N),H_2(N))$ be the element corresponding to the identity. Note that $\alpha=\tilde \alpha_*\id$. By naturality of the Hochschild-Serre spectral sequence we have a commutative diagram:
		\[
		\begin{tikzcd}
		{H^2(N,M_{univ})^\pi} \arrow[d, "\tilde \alpha_*"] \arrow[r, "d_2^{0,2}"] & {H^2(\pi,H^1(N,M_{univ}))} \arrow[d, "\tilde \alpha_*"] \\
		{H^2(N,M)^\pi} \arrow[r, "d_2^{0,2}"]                               & {H^2(\pi,H^1(N,M))}                              
		\end{tikzcd}
		\]
		
		Let $v_2(N) \coloneqq d(\id)\in {H^2(\pi,H^1(N,M_{univ}))} $. The statement follows by a diagram chase.
	\end{proof}
	
	The following proposition is a special case of \cite[Example 3]{CV1}, where, however, a complete proof is not presented (only hinted at).
	
	\begin{proposition}[Charlap-Vasquez]\label{Prop1}
		Assume that $\pi\cong C_2$ is the cyclic group of order $2$, then the class $v_2(N)$ is trivial.
	\end{proposition}
	
	
	\begin{proof}
		Any $C_2$-module $N$ decomposes as a direct sum of the following three modules:
		\begin{itemize}
			\item $\Z$, with trivial $C_2$-action;
			\item $\Z(1)$, i.e.\ the abstract group $\Z$ endowed with the involution $n \mapsto -n$;
			\item $\Ind^{C_2}_{<id>}\Z$.
		\end{itemize}
		Since $v_2(N_1\oplus N_2)=v_2(N_1) \oplus v_2(N_2)$ (see \cite[Theorem 7]{CV2}), we may assume without loss of generality that $N$ is one of the three $C_2$-modules listed above, and thus has rank $\leq 2$. 
		If $N$ has rank $1$, then $H_2(N)=0$, and so $v_2=0$ trivially. On the other hand, if $N$ has rank $2$, then Sah \cite[Theorem 12, section I]{Sah1} shows that $v_2=0$.
	\end{proof}

	\begin{proof}[Proof of Theorem \ref{Thm3}]
            The Hochschild-Serre spectral sequence of $T/\R$ gives an exact sequence
            \[
            H^2(T,\mu_{\infty}) \to {H^2(\bar T,\mu_{\infty})^{\Gal(\C/\R)}} \to[{d_{2,T}^{0,2}}]  {H^2(\Gal(\C/\R),H^1(\bar T,\mu_{\infty}))},
            \]
            where ${d_{2,T}^{0,2}}$ denotes the spectral sequence differential. Since the Picard group of tori is torsion, a Kummer sequence argument shows that  $H^2(T,\mu_{\infty})=H^2(T,\G_m)=\Br T$ and the analog for $\bar T$ (see \cite[Théorème 3.1]{BrauerII}). It thus suffices to show that  ${d_{2,T}^{0,2}}$ vanishes.
		Consider the short exact sequence \cite[Proposition 5.6.1]{szamuely}
		\begin{equation}\label{Eq5}
		1 \to \pi_1(T_{\C}) \to \pi_1(T) \to \Gal(\C/\R),
		\end{equation}
            where, letting $e: \Spec \R \to T$ be the the identity of $T$ and  $\bar e: \Spec \C \to T$ be the corresponding geometric, $\pi_1(T)$ and $\pi_1(T_{\C})$ denote the étale fundamental groups $\pi_1(T,\bar e)$ and $\pi_1(T_{\C},\bar e)$. The point $e$ induces by functoriality a section $\Gal(\C/\R) =\pi_1(e,\bar e) \to \pi_1(T)$.
            
            As explained in more detail in a general setting in the appendix, there are maps (arising from natural transformations of derived functors):
            \[
            N_T:H^n(\pi_1(T),\mu_{\infty}) \to H^n(T, \mu_{\infty}), \ \  N_{T_{\C}}:H^n(\pi_1(T_{\C}),\mu_{\infty}) \to H^n(T_{\C}, \mu_{\infty}),
            \]
            for all $n\geq 0$. The map $N_{T_{\C}}$ is an isomorphism by \cite[Proposition 3.4]{GP} (in fact, $N_T$ is an isomorphism as well, but we do not need this, and it is not proven in {\em loc.cit.}).
            Using Corollary \ref{Cor2}, we may compare the Hochschild-Serre spectral sequences of the group extension \eqref{Eq5} and of $T/\R$, obtaining a commutative diagram:
		\[
            \begin{tikzcd}[column sep = huge]
                {H^2(\pi_1(\bar T),\mu_{\infty})^{\Gal(\C/\R)}} \arrow[d, "\sim"] \arrow[r, "{d^{0,2}_{2,\pi_1(T)}}"] & {H^2(\Gal(\C/\R),H^1(\pi_1(\bar T),\mu_{\infty}))} \arrow[d, "\sim"] \\
                {H^2(\bar T,\mu_{\infty})^{\Gal(\C/\R)}} \arrow[r, "{d^{0,2}_{2,T}}"]                                       & {H^2(\Gal(\C/\R),H^1(\bar T,\mu_{\infty}))}.
            \end{tikzcd}
		\]
            It remains to show that $d_{2,\pi_1(T)}^{0,2}$ vanishes. We have a canonical identification \cite[Corollary 2.10]{GP}
		\[
		\pi_1(T_{\C})= X_*(T) \otimes_{\Z}\widehat \Z(1) = (X_*(T) (1) )\otimes \widehat \Z,
		\]
		where $X_*(T)(1) \coloneqq X_*(T) \otimes \Z(1)$ (note that the Tate twist $\Z(1)$ is defined only because our base field is $\R$). Letting $G \coloneqq (X_*(T)(1)) \rtimes \Gal(\C/\R)$, we may compare the profinite split extension \eqref{Eq5} to a discrete one, i.e.\  we have a commutative diagram:
		\begin{equation*}
		\begin{tikzcd}
		1 \arrow[r] & X_*(T)(1) \arrow[d, hook] \arrow[r]      & G \arrow[d, hook] \arrow[r]          & \Gal(\C/\R) \arrow[d, "="] \arrow[r] & 1 \\
		1 \arrow[r] & {\pi_1(T_{\C})} \arrow[r] & {\pi_1(T) } \arrow[r] & \Gal(\C/\R) \arrow[r]                & 1.
		\end{tikzcd}
		\end{equation*}
		Comparing the Hochschild-Serre spectral sequences of these two group extensions, we get a commutative diagram
		\[
		\begin{tikzcd}
		{H^2(X_*(T)(1),\mu_{\infty})^{\Gal(\C/\R)}} \arrow[r, "{d_{2,G}^{0,2}}"] \arrow[d, "\sim", leftarrow] & {H^2(\Gal(\C/\R),H^1(X_*(T)(1),\mu_{\infty}))} \arrow[d, "\sim", leftarrow] \\
		{H^2(\pi_1(T_{\C}),\mu_{\infty})^{\Gal(\C/\R)}} \arrow[r, "{d_{2,\pi_1(T)}^{0,2}}"]               & {H^2(\Gal(\C/\R),H^1(\pi_1(T_{\C}),\mu_{\infty}))},
		\end{tikzcd}
		\]
		where the vertical maps are isomorphisms because $H^n(\pi_1(T_{\C}),\mu_{\infty})=\Hom(\Lambda^n(X_*(T) \otimes \widehat \Z(1)),\mu_{\infty})=\Hom(\Lambda^n(X_*(T)(1)),\mu_{\infty})$ for all $n \geq 1$ \cite[Theorem V.6.4, Exercise III.1.3]{Brown}. The sought vanishing now follows from the Charlap-Vesquez theory (see Proposition \ref{Prop1}).
	\end{proof}
	
	\begin{proof}[{Proof of Theorem \ref{Thm1}}]
        \noindent (i). This follows from Theorem \ref{Thm2}.

        \noindent (ii). As mentioned in the introduction, the exact sequence
	\[
	\Br T \to (\Br \bar T)^{\Gamma_k} \to[d_3^{0,2}] H^3(k,\bar k[T]^*)
	\]
	induced by the Hochschild-Serre spectral sequence (note that $\Pic \bar T=0$ and thus the second page differential $d_2^{0,2}$ vanishes), proves the theorem when $k$ is a local non-archimedean field (as the strict cohomological dimension of these fields is $2$ \cite[Theorem 10.6]{HarariBook}), and it reduces the case of number fields to the case of $\R$ because of the Poitou-Tate isomorphism $H^3(k,\ok[T]^*)= \bigoplus_{\substack{v \in M_K\\v \text{ real}}}H^3(k_v,\ok_v[T]^*)$ \cite[Theorem 17.3(a)]{HarariBook}, in which case the result follows from Theorem \ref{Thm3}.
	\end{proof}
	
	\appendix
	
	\section{Functoriality of Grothendieck's spectral sequence}\label{App}
	
	The purpose of this appendix is to prove the naturality of the Grothendieck spectral sequence with respect to exact functors. This is merely an exercise in homological algebra, but the author was unable to find a precise statement in the literature. We borrow from \cite[Sec.\ 10]{Weibel} our notation for derived categories and definition of derived functors:
	
	\vskip3mm
	
	{\bf Notation for derived categories.} For an abelian category $\cA$, we denote by $\operatorname{Ch}(\cA)$ the category of chain complexes, by $\cK(\cA)$ the category whose objects are chain complexes of $\cA$ and morphisms are chain homotopy equivalence classes of morphisms in $\operatorname{Ch}(\cA)$, and by $\cD(\cA)$ the derived category of $\cA$, i.e.\ the localization of $\cK(\cA)$ with respect to quasi-isomorphisms. We denote by $\operatorname{Ch}^+(\cA) \subset \operatorname{Ch}(\cA), \cK^+(\cA) \subset \cK(\cA), \cD^+(\cA) \subset \cD(\cA)$ the full subcategories where complexes are bounded from below. 

        \vskip3mm

        {\bf Compositions.} To lighten the notation:
        \begin{itemize} \vspace{-0.9mm} \setlength\itemsep{0.1em}
            \item  we denote compositions of functors $G \circ F$ by $GF$;
            \item when we have a composition of functors $GF$ and a natural transformation $N:F \to F'$, we denote the induced natural transformation $GF \to GF'$ again by $N$. (Same for $G$ and multiple compositions.)
        \end{itemize}
    	
	\vskip3mm
	
	{\bf Derived functors. }Let  $F:\cA \to \cB$ be a left exact additive morphism  between abelian categories. With a slight abuse of notation, we denote with $F:\cK^+(\cA) \to \cK^+(\cB)$ the functor $(C^{\bullet}) \mapsto (F(C^{\bullet}))$. A {\em right derived functor} of $F$ is a morphism $RF:\cD^+(\cA) \rightarrow$ $\cD^+(\cB)$ of triangulated categories, together with a natural transformation $\xi$ from $q F: \cK^+(\cA) \rightarrow {\cK^+( B )} \rightarrow \cD^+(\cB)$ to $(RF) q: \cK^+(\cA) \rightarrow \cD^+(\cA) \rightarrow \cD^+(\cB)$ which is universal in the sense that if ${G}: \cD^+(\cA) \rightarrow \cD^+(\mathcal{B})$ is another morphism equipped with a natural transformation $\zeta: q F \Rightarrow	G q$, then there exists a unique natural transformation $\eta: {RF} \Rightarrow	{G}$ such that $\zeta$ is the composition $qF \xRightarrow{\xi} (RF)q \xRightarrow{\eta} Gq$. When there might be risk of misinterpretation, we write $R(F)$ instead of $RF$.

        \vskip3mm

        {\bf Compositions of derived functors.} Let $F:\cA \to \cB, G:\cB \to \cC$ be additive left exact functors among abelian categories. Assume that both $F, G$ and $GF$ posses derived functors, then it follows from the universal property of $R(FG)$ that there exists a unique natural transformation $\zeta=\zeta_{G,F}:R(GF) \to R(G)R(F)$ of functors from $\cD^+(\cA)$ to $\cD^+(\cC)$ such that $q(GF) \xRightarrow{\xi_{GF}} R(GF)q$ coincides with $qGF \xRightarrow{\xi_F\xi_G} R(G)R(F)q\xRightarrow{\zeta} R(GF)q$ \cite[Section 10.8]{Weibel}. When $F$ sends injectives to $G$-acyclics, $\zeta$ is an isomorphism\cite[Theorem 10.8.2]{Weibel}, inducing an identification $R(GF)=R(G)R(F)$.

        \vskip3mm
	
	Let $\cA, \cB, \cA', \cB'$ be abelian categories, and let 
	\[\begin{tikzcd}
	\cA & \cB \\
	{\cA'} & {\cB'}
	\arrow["F", from=1-1, to=1-2]
	\arrow["{\kappa_{\cA}}"', from=1-1, to=2-1]
	\arrow[Rightarrow, "N"{description}, from=1-2, to=2-1]
	\arrow["{\kappa_{\cB}}", from=1-2, to=2-2]
	\arrow["{F'}"', from=2-1, to=2-2]
	\end{tikzcd}\]
	be a $2$-commutative diagram, in the sense that the double arrow denotes a natural transformation $N:\kappa_{\cB}   F \Rightarrow	F'   \kappa_{\cA}$. Assume that the right derived functors of $F,F', \kappa_{\cA}$ and $\kappa_{\cA'}$ exist. 
	
	\begin{lemma}
		Assume that $\cA$ and $\cA'$ have enough injectives, and that $F$ sends injectives into $\kappa_{\cB}$-acyclics. Then  there is a unique natural transformation $R(F)R(\kappa_B) \to R({\kappa_{\cA}}  ) R(F')$ (which we still denote with $N$ with an innocous abuse of notation) as in the following diagram:
		\[\begin{tikzcd}
		{\cD^+(\cA)} & {\cD^+(\cB)} \\
		{\cD^+(\cA')} & {\cD^+(\cB')}
		\arrow["RF", from=1-1, to=1-2]
		\arrow["{R\kappa_{\cA}}"', from=1-1, to=2-1]
		\arrow[Rightarrow, "N"{description}, from=1-2, to=2-1]
		\arrow["{R\kappa_{\cB}}", from=1-2, to=2-2]
		\arrow["{RF'}"', from=2-1, to=2-2]
		\end{tikzcd}\]
		that $2$-commutes with the diagram
		\[\begin{tikzcd}
		{\cK^+(\cA)} & {\cK^+(\cB)} \\
		{\cK^+(\cA')} & {\cK^+(\cB')}
		\arrow["F", from=1-1, to=1-2]
		\arrow["{\kappa_{\cA}}"', from=1-1, to=2-1]
		\arrow[Rightarrow, "N"{description},from=1-2, to=2-1]
		\arrow["{\kappa_{\cB}}", from=1-2, to=2-2]
		\arrow["{F'}"', from=2-1, to=2-2]
		\end{tikzcd}\]
		and the localizations $q$ and natural transformations $\xi$.
	\end{lemma}
	
	With the last sentence, we mean here that the cubic diagram
	\begin{equation}\label{Eq8}
	\begin{tikzcd}
	& {\cD^+(\cA)} & {\cD^+(\cB)} \\
	{\cD^+(\cA')} &&& {\cD^+(\cB')} \\
	& {\cK^+(\cA)} & {\cK^+(\cB)} \\
	{\cK^+(\cA')} &&& {\cK^+(\cB')}
	\arrow["RF", from=1-2, to=1-3]
	\arrow["{R\kappa_{\cA}}"', from=1-2, to=2-1]
	\arrow["{R\kappa_{\cB}}", from=1-3, to=2-4]
	\arrow["q"{ pos=0.2}, from=3-2, to=1-2]
	\arrow["F", from=3-2, to=3-3]
	\arrow["{\kappa_{\cA}}"', from=3-2, to=4-1]
	\arrow["q"{ pos=0.2}, from=3-3, to=1-3]
	\arrow[Rightarrow, from=3-3, to=4-1]
	\arrow["{\kappa_{\cB}}", from=3-3, to=4-4]
	\arrow["q", from=4-1, to=2-1]
	\arrow["{F'}", from=4-1, to=4-4]
	\arrow["q", from=4-4, to=2-4]
	\arrow[Rightarrow, crossing over, from=1-3, to=2-1]
	\arrow["{RF'}", crossing over, from=2-1, to=2-4]
	\end{tikzcd}
	\end{equation}
	where the (undrawn) natural transformations $\xi$ appear on the lateral faces, is $2$-commutative. Equivalently, that the following diagram commutes:
        \[\begin{tikzcd}
        	{q\kappa_{\cB}F} & {qF'\kappa_{\cA}} \\
        	{R(\kappa_{\cB})R(F)q} & {R(F')R(\kappa_{\cA})q}
        	\arrow["N", from=1-1, to=1-2]
        	\arrow["{\xi_{F}\xi_{\kappa_{\cB}}}"', from=1-1, to=2-1]
        	\arrow["{\xi_{F'}\xi{\kappa_{\cA}}}", from=1-2, to=2-2]
        	\arrow["N"', from=2-1, to=2-2]
        \end{tikzcd}\]
	
	\begin{proof}
		Since $F$ sends injectives in $\kappa_{\cB}$-acyclics, the natural transformation $\zeta_{\kappa_{\cB},F}:R(\kappa_{\cB}   F) \linebreak \Rightarrow	R(\kappa_{\cB})R(F)$  is an isomorphism and it $2$-commutes with $q$'s and $\xi$'s. Thus we may equivalently prove that there exists a unique natural transformation $R(\kappa_{\cB}F) \Rightarrow R(\kappa_{\cA})R(F')$ that $2$-commutes with $q$'s and $\xi$'s and the natural transformation $\kappa_{\cB}F \Rightarrow \kappa_{\cA}F':\cK^+(\cA) \to \cK^+(\cB')$. This follows from the universal property of the derived functor $R(F\kappa_{\cB})$.
	\end{proof}
	
	Consider now a $2$-commutative diagram
        \begin{equation}\label{Eq7}\begin{tikzcd}[column sep=large]
            \cA & \cB & \cC \\
            {\cA'} & {\cB'} & {\cC'}
            x+\arrow["F", from=1-1, to=1-2]
            \arrow["{\kappa_{\cA}}"', from=1-1, to=2-1]
            \arrow["G", from=1-2, to=1-3]
            \arrow["{N_F}"{description}, Rightarrow, from=1-2, to=2-1]
            \arrow["{\kappa_{\cB}}"', from=1-2, to=2-2]
            \arrow["{N_G}"{description}, Rightarrow, from=1-3, to=2-2]
            \arrow["{\kappa_{\cC}}"', from=1-3, to=2-3]
            \arrow["{F'}"', from=2-1, to=2-2]
            \arrow["{G'}"', from=2-2, to=2-3]
            \end{tikzcd}
        \end{equation}
	where $\cA, \cB, \cC$ and  $\cA', \cB', \cC'$ are abelian categories of which $\cA, \cB$ and  $\cA', \cB'$ have enough injectives. All the functors (straight lines) appearing in the diagram are additive and left exact, and suppose that $F$ (resp.\ $F'$) sends injectives in $\kappa_{\cB}$-acyclics (resp.\ $\kappa_{\cC}$-acyclics). The double lines denote natural transformations $N_F:\kappa_{\cB}   F \Rightarrow	F'   \kappa_{\cA}$ and  $N_G:\kappa_{\cC}   G \Rightarrow	G'   \kappa_{\cB}$. The composition of $N_F$ and $N_G$ defines a natural transformation
	\[
	N_{GF}\coloneqq N_F   N_G:\kappa_{\cC} GF \xRightarrow{N_G} G'\kappa_{\cB}F \xRightarrow{N_F} G'F' \kappa_{\cA}.
	\]
	Note that all functors in the diagram possess derived functors, as do the compositions $GF$ and $G'F'$. 
 
	\begin{proposition}\label{Prop2}
		The diagram
		\begin{equation}\label{Eq9}
		\begin{tikzcd}
		& {\cD^+(\cB)} \\
		{\cD^+(\cA)} && {\cD^+(\cC)} \\
		{\cD^+(\cA')} && {\cD^+(\cC')} \\
		& {\cD^+(\cB')}
		\arrow["RG", from=1-2, to=2-3]
		\arrow["{N_F}"{description, pos=0.7}, Rightarrow, from=1-2, to=3-1]
		\arrow["{\kappa_{\cB}}"{description, pos=0.1}, from=1-2, to=4-2]
		\arrow["RF", from=2-1, to=1-2]
		\arrow["{\kappa_{\cA}}"', from=2-1, to=3-1]
		\arrow["{\kappa_{\cC}}", from=2-3, to=3-3]
		\arrow["{N_G}"{description, pos=0.7}, Rightarrow, from=2-3, to=4-2]
		\arrow["{RF'}"', from=3-1, to=4-2]
		\arrow["{RG'}"', from=4-2, to=3-3]
		\arrow["{N_{GF}}"{description, pos=0.8}, crossing over, Rightarrow, from=2-3, to=3-1]
		\arrow["{R(G'F')}"{description, pos=0.3}, crossing over, from=3-1, to=3-3]
		\arrow["{R(GF)}"{description, pos=0.8}, crossing over, from=2-1, to=2-3],
		\end{tikzcd}
		\end{equation}
		where the (undrawn) natural tranformations $\zeta:R(GF) \Rightarrow R(G)R(F)$ and $\zeta:R(G'F') \Rightarrow R(G')R(F')$ appear at the top and bottom triangles,
		is $2$-commutative.
	\end{proposition}
	
	Equivalently, the above proposition says that the following diagram commutes:
        \begin{equation}\label{Eq19}
            \begin{tikzcd}
            R(\kappa_{\cC})R(GF) \arrow[Rightarrow,d, "{\zeta_{G,F}}"] \arrow[Rightarrow,rr, "N_{GF}"] &                                           & R(G'F')R(\kappa_{\cA}) \arrow[Rightarrow,d, "{\zeta_{G',F'}}"] \\
            R(\kappa_{\cC})R(G)R(F) \arrow[Rightarrow,r, "N_G"]                             & R(G')R(\kappa_{\cB})R(F) \arrow[Rightarrow,r, "N_F"] & R(G')R(F')R(\kappa_{\cA}).                        
            \end{tikzcd}
        \end{equation}
	
	\begin{proof}
		The universal property of $R(\kappa_{\cC}GF)$ shows that the two natural transformations $R(\kappa_{\cC}GF) \cong R(\kappa_{\cC})R(GF) \Rightarrow R(G')R(F')R(\kappa_{\cA})$ arising from diagram \eqref{Eq19} coincide.
	\end{proof}

        We further assume now that $F$ and $F'$ send injectives in injectives, and that all the the functors $\kappa$ ($\kappa_{\cA}, \kappa_{\cB},$ and $\kappa_{\cC}$) are exact. In particular, the functors $\kappa$ are their own derived functors, i.e.\ $R\kappa=\kappa$.
	
	Let $E_r^{p,q}(A)$ be the Grothendieck spectral sequence of the composition $GF$ for an object $A \in \cA$ (see \cite[Sec.s 5.8,10.8]{Weibel}). Recall that $E_2^{p,q}=R^pG(R^qF(A))$, and that the sequence converges to $R^{p+q}(GF)(A)$. For $A \in \cA'$, we denote by $\tilde E_r^{p,q}(A)$ the Grothendieck spectral sequence of the composition $G'F'$. 
	
	\begin{corollary}\label{Cor1}
		For all $A \in \cA$, there is a morphism of spectral sequences:
		\[
		\kappa_{\cC}(E_r^{p,q}(A)) \to \tilde E_r^{p,q}(\kappa_{\cA}(A)),
		\]
		natural in $A$, that coincides with
		\[
		(\kappa_{\cC} R^pG R^qF)(A) \to[N_FN_G] (R^pG R^qF \kappa_{\cA})(A)
		\]
		on the second page and such that the diagram
		\[\begin{tikzcd}
		{\kappa_{\cC}(E_r^{p,q}(A))} & {\kappa_{\cC}R^{p+q}(GF)(A)} \\
		{\tilde E_r^{p,q}(\kappa_{\cA}(A))} & {R^{p+q}(G'F')(\kappa_{\cA}(A))}
		\arrow[Rightarrow, from=1-1, to=1-2]
		\arrow[from=1-1, to=2-1]
		\arrow["N_{GF}",from=1-2, to=2-2]
		\arrow[Rightarrow, from=2-1, to=2-2]
		\end{tikzcd}\]
		commutes.
	\end{corollary}
	
	We use the following in the proof:
	\begin{definition}
		Let $\cA$ be an abelian category and $C^{\bullet} \in \operatorname{Ch}^+(\cA)$. A {\em resolution} $K^{\bullet\bullet}$ of $C^{\bullet}$ is an upper half-plane double complex in $\cA$ endowed with an augmentation map $\epsilon:C^{\bullet}\to K^{\bullet,0}$ such that, denoting by 
        \begin{itemize}
            \item $\mathrm{B}^{\bullet,\bullet}(K)$ the horizontal coborders $\mathrm{B}_1^{p,q}(K) \coloneqq \im (\delta_{hor}:K^{p-1,q} \to K^{p,q})$, and 
            \item $\mathrm{H}^p(K^{\bullet\bullet})$ the horizontal cohomology  $\mathrm{H}_1^{p,q}(K) \coloneqq \Ker  (\delta_{hor}:K^{p,q} \to K^{p+1,q})/ \im (\delta_{hor}:K^{p-1,q} \to K^{p,q})$,
        \end{itemize}
        we have, for every integer $p$:
		\begin{itemize}
			\item the map $B_1^p(\epsilon): B^p(C)[0] \to B_1^{p,\bullet}(K)$ is a resolution;
			\item the map $H_1^p(\epsilon): H^p(C)[0] \to H_1^{p,\bullet}(K )$ is a resolution.
		\end{itemize}
	\end{definition}

        Recall that a {\em} resolution of an object $A \in \cA$ is a quasi-isomorphism $A[0] \to K^{\bullet}$, where $K^{\bullet}$ is a chain complex concentrated in positive degree \cite[Definition 2.3.5]{Weibel}.
	A {\em Cartan-Eilenberg resolution} of a complex $C^{\bullet}$ is a resolution $C^{\bullet} \to J^{\bullet \bullet}$, where all the entries, the horizontal coborders and the horizontal cohomologies are injective \cite[Chapter XVII]{CE} \cite[Definition 5.7.1]{Weibel}.

	\begin{lemma}\label{Lem3}
		Let $C^{\bullet} \to \tilde C^{\bullet}$ be a morphism of positively bounded complexes in an abelian category, $C^{\bullet} \to K^{\bullet\bullet}$ be a resolution,
		and $\tilde C^{\bullet} \to \tilde J^{\bullet\bullet}$ be a Cartan-Eilenberg resolution. Then there exists a morphism of double complexes $\phi:K^{\bullet\bullet} \to J^{\bullet\bullet}$ that fits in a commutative diagram
		\[
		\begin{tikzcd}
		K^{\bullet\bullet} \arrow[r, "\phi", dotted] & \tilde{J}^{\bullet\bullet}   \\
		C^{\bullet} \arrow[r] \arrow[u]              & \tilde C^{\bullet} \arrow[u]
		\end{tikzcd}
		\]
	\end{lemma}
	\begin{proof}
		The same proof of \cite[Proposition XVII.1.2]{CE} proves this statement.
	\end{proof}
	
	\begin{proof}[Proof of Corollary \ref{Cor1}]
		Let $RF(A) \to J^{\bullet\bullet}$ and $RF'(\kappa_{\cA}A) \to \tilde J^{\bullet\bullet}$ be Cartan-Eilenberg resolutions in $\cA$ and $\cA'$. Note that, since $\kappa_{\cB}$ is exact, $\kappa_{\cB}RF(A) \to \kappa_{\cB}J^{\bullet\bullet}$ is still a resolution (though not necessarily Cartan-Eilenberg). 
		Thus by Lemma \ref{Lem3}, there exists a morphism $\phi:\kappa_{\cB}J^{\bullet \bullet} \dashrightarrow \tilde J^{\bullet\bullet}$ fitting in a commutative diagram of complexes:
		\[
		\begin{tikzcd}
		\kappa_{\cB}J^{\bullet\bullet} \arrow[r, "\phi", dotted] & \tilde{J}^{\bullet\bullet}   \\
		\kappa_{\cB}RF(A) \arrow[r, "N_F"] \arrow[u]            & RF'(\kappa_{\cA}A) \arrow[u]
		\end{tikzcd}.
		\]
		Applying $G'$ to the diagram, and precomposing the rows with the natural transformation $N_G:\kappa_{\cC}G \to G'\kappa_{\cB}$,
		induces a commutative diagram
		\begin{equation}\label{Eq13}
		\begin{tikzcd}[column sep=huge]
		\kappa_{\cC}G(J^{\bullet\bullet}) \arrow[r, "\phi \circ N_G"] & G'(\tilde{J}^{\bullet\bullet})   \\
		\kappa_{\cC}G(RF(A)) \arrow[u] \arrow[r, "N_F   N_G"]     & G'(RF'(\kappa_{\cA}A)) \arrow[u].
		\end{tikzcd}
		\end{equation}  

        We may assume that the derived functors $RH$ for  $H=F,G,F',G'$ are constructed by sending a chain complex $A^{\bullet}$ into the chain complex $H(I(A)^{\bullet})$, where $A^{\bullet} \to I(A)^{\bullet}$ is a quasi-isomorphism, and $A^{\bullet}=I(A)^{\bullet}$ when $A^{\bullet}$ is made of injectives \cite[Section 10.5]{Weibel}. In particular, we have
        \begin{gather*}
            G(RF(A))=RG(RF(A))=R(GF)(A), \\
            G'(RF'(\kappa_{\cA}A))=RG'(RF'(\kappa_{\cA}A))=R(G'F')(\kappa_{\cA}A),
        \end{gather*}
         and, under these identifications, Proposition \ref{Prop2} gives that:
        \begin{center}
            $(\star)$  the natural transformation $N_FN_G:\kappa_{\cC}G(RF(A)) \to G'(RF'(\kappa_{\cA}A))$ coincides with $N_{GF}:\kappa_{\cC}R(GF)(A) \to R(G'F')(\kappa_{\cA}A)$.
        \end{center}
        The Grothendieck spectral sequence $E_r^{p,q}(A)$ (resp.\ $\tilde E_r^{p,q}(\kappa_{\cA}A)$) is the second spectral sequence of the complex $G(J^{\bullet\bullet})$ (resp.\ $G'(\tilde J^{\bullet\bullet})$), which converges to 
        \[
        H^{p+q}(G(RF(A))=H^{p+q}(RG(RF(A)))=H^{p+q}(R(GF)(A))
        \]
        (resp.\ $H^{p+q}(G'(RF'(\kappa_{\cA}A)))=H^{p+q}(R(G'F')(\kappa_{\cA}A))$) \cite[Corollary 10.8.3]{Weibel}. The sought map of spectral sequences $\kappa_{\cC}E_r^{p,q} \to \tilde E_r^{p,q}$ is induced by the morphism of double complexes $\phi \circ N_G$. That the abutments commute with $N_{GF}$ follows from $(\star)$.
	\end{proof}
	
	\subsection{Compatibility of Leray and Hocschild-Serre spectral sequences}

        Recall that, for any connected scheme $S$ and geometric point $\bar s$ on it, we have an equivalence of categories (see \cite[Tag 0GIY]{stacks-project}):
	$$
	\left\{\begin{array}{c}
	\text { finite } \\ \pi_1(S, \bar{s})-\text{sets } 
	\end{array}\right\}
	\leftrightarrow
	\left\{\begin{array}{c}
	\text { locally constant sheaves} \\
	\text {  of finite sets on } S_{\text {étale }}
	\end{array}\right\} .
	$$
	Let $A \mapsto \tilde A$ be the functor from left to right. We extend this functor to a functor
	$$
	\left\{\begin{array}{c}
	\text { discrete } \\ \pi_1(S, \bar{s})-\text{sets } 
	\end{array}\right\}
	\rightarrow
	\left\{\begin{array}{c}
	\text { sheaves} \\
	\text {  of finite sets on } S_{\text {étale }}
	\end{array}\right\} ,
	$$
	by sending a coproduct $\sqcup_{i \in I} A_i$, where the $A_i$ are finite $\pi_1(S,\bar s)$-sets, into the étale sheaf $\sqcup_{i \in I} \tilde{A_i}$. The functor above restricts to an exact functor
	\begin{equation}\label{Eq12}
	\left\{\begin{array}{c}
	\text { discrete } \\ \pi_1(S, \bar{s})-\text{modules } 
	\end{array}\right\}
	\rightarrow
	\left\{\begin{array}{c}
	\text {sheaves} \\
	\text {  of abelian groups on } S_{\text {étale }}
	\end{array}\right\},
	\end{equation}
	which we denote by $M \mapsto \tilde M$. Note that the global sections $\Gamma_S(\tilde M)$ are equal to $M^{\pi_1(S,\bar s)}$ (see \cite[Sec.I.5]{LECcompleto}), we thus have a commutative diagram
\[\begin{tikzcd}[column sep=huge]
	{\pi_1(S,\bar s)-\operatorname{Mod}} & \Ab \\
	{\Sh_S} & \Ab
	\arrow["{\star^{\pi_1(S,\bar s)}}", from=1-1, to=1-2]
	\arrow[from=1-1, to=2-1]
	\arrow["{=}", no head, from=1-2, to=2-2]
	\arrow["{\Gamma_S}"', from=2-1, to=2-2],
\end{tikzcd}\]
        where the first vertical morphism is the exact functor \eqref{Eq12}. This commutative diagram induces a natural transformation of derived functors of the rows, which itself induces natural transformations on cohomology:
        \begin{equation}\label{Eq17}
            N_S:H^n(\pi_1(S,\bar s),M) \to H^n(S,\tilde M).
        \end{equation}

        Let now $k$ be a perfect field, $f:X \to \Spec k$ be a geometrically integral normal $k$-variety, and $\bar x: \Spec \ok \to X$ be a geometric point. We have a short exact sequence of étale fundamental groups \cite[Proposition 5.6.1]{szamuely}:
        \begin{equation}\label{Eq18}
             1 \to \pi_1(X_{\ok},\bar x) \to \pi_1(X, \bar x) \to \pi_1(\Spec k, \Spec \ok)=\Gamma_k \to 1.
        \end{equation}

	
	Let $M$ be a discrete $\pi_1(X,\bar x)$-module, $\tilde M$ the associated étale sheaf on $X$, and let 
	\[
	E_r^{p,q}, \tilde E_r^{p,q}
	\]
	be the Hochschild-Serre spectral sequences
        \begin{align*}
            &E_2^{p,q}=H^p(\Gamma_k,H^q(\pi_1(X_{\ok},\bar x),M)) \Rightarrow H^{p+q}(\pi_1(X,\bar x),M), \\
            &\tilde E_2^{p,q}=H^p(\Gamma_k,H^q(X_{\ok},\tilde M))  \Rightarrow H^{p+q}(X,\tilde M),
        \end{align*}
	associated to the group extension \eqref{Eq18} and to $X/k$, respectively.
	
	\begin{corollary}\label{Cor2}
		There is a morphism of spectral sequences
		\[
		N_{X_{\ok}}:E_r^{p,q} \to \tilde E_r^{p,q},
		\]
		natural in $M$, that coincides on the second page with the morphism
		\[
		H^p(\Gamma_k,H^q(\pi_1(X_{\ok},\bar x),M)) \to H^p(\Gamma_k,H^q(X_{\ok},\tilde M))
		\]
		induced by the natural transformation $N_{X_{\ok}}$, and such that the convergence diagram
		\[\begin{tikzcd}
		{E_r^{p,q}} & {H^{p+q}(\pi_1(X,\bar x),M)} \\
		{\tilde E_r^{p,q}} & {H^{p+q}(X,\tilde M)}
		\arrow[Rightarrow, from=1-1, to=1-2]
		\arrow[from=1-1, to=2-1, "N_{X_{\ok}}"]
		\arrow[from=1-2, to=2-2, "N_X"]
		\arrow[Rightarrow, from=2-1, to=2-2]
		\end{tikzcd}\]
		commutes.
	\end{corollary}
	\begin{proof}
        We identify the categories $\Sh_k$ and $\Gamma_k-\operatorname{Mod}$ \cite[Theorem 1.9]{LECcompleto}, and apply Corollary \ref{Cor1} to the commutative diagram:
        \[\begin{tikzcd}
	{\pi_1(X,\bar x)-\operatorname{Mod}} & {\Gamma_k-\operatorname{Mod}} & {\cA b} \\
	{\Sh_X} & {\Sh_k=\Gamma_k-\operatorname{Mod}} & {\cA b} 
	\arrow["{\star^{\pi_1(X,\bar x)}}", from=1-1, to=1-2]
	\arrow[from=1-1, to=2-1]
	\arrow["{\star^{\Gamma_k}}", from=1-2, to=1-3]
	\arrow["{=}", from=1-2, to=2-2]
	\arrow["{=}", from=1-3, to=2-3]
	\arrow["{f_*}", from=2-1, to=2-2]
	\arrow["{\star^{\Gamma_k}}", from=2-2, to=2-3]
        \end{tikzcd}\]
		where $\star^H$ denotes the functor $M \mapsto M^H$, and the first two vertical functors are as defined through \eqref{Eq12}. The composition on the first row is the functor $\star^{\pi_1(X,\bar x)}$, and its associated Grothendieck spectral sequence is the Hochschild-Serre spectral sequence of the composition \eqref{Eq18} \cite[p.105]{LECcompleto}, while the composition on the second row is the global sections functor $\Gamma_X$ and its associated spectral sequence is the Hochschild-Serre spectral sequence of $X/k$ \cite[Sec.6.8]{Weibel}. 
	\end{proof}
	
	\bibliographystyle{abbrv}      
	\bibliography{homspaces}

\begin{thebibliography}{10}

\bibitem{Brown}
K.~S. Brown.
\newblock {\em Cohomology of groups}, volume~87 of {\em Graduate Texts in
  Mathematics}.
\newblock Springer-Verlag, New York, 1994.
\newblock Corrected reprint of the 1982 original.

\bibitem{CE}
H.~Cartan and S.~Eilenberg.
\newblock {\em Homological algebra}.
\newblock Princeton Landmarks in Mathematics. Princeton University Press,
  Princeton, NJ, 1999.
\newblock With an appendix by David A. Buchsbaum, Reprint of the 1956 original.

\bibitem{CV1}
L.~S. Charlap and A.~T. Vasquez.
\newblock The cohomology of group extensions.
\newblock {\em Trans. Amer. Math. Soc.}, 124:24--40, 1966.

\bibitem{CV2}
L.~S. Charlap and A.~T. Vasquez.
\newblock Characteristic classes for modules over groups. {I}.
\newblock {\em Trans. Amer. Math. Soc.}, 137:533--549, 1969.

\bibitem{BGbook}
J.-L. Colliot-Th\'{e}l\`ene and A.~N. Skorobogatov.
\newblock {\em The {B}rauer-{G}rothendieck group}, volume~71 of {\em Ergebnisse
  der Mathematik und ihrer Grenzgebiete. 3. Folge. A Series of Modern Surveys
  in Mathematics}.
\newblock Springer, Cham, [2021] \copyright 2021.

\bibitem{GP}
P.~Gille and A.~Pianzola.
\newblock Isotriviality and \'{e}tale cohomology of {L}aurent polynomial rings.
\newblock {\em J. Pure Appl. Algebra}, 212(4):780--800, 2008.

\bibitem{BrauerII}
A.~Grothendieck.
\newblock Le groupe de {B}rauer. {II}. {T}h\'eorie cohomologique [{MR}0244270
  (39 \#5586b)].
\newblock In {\em S\'eminaire {B}ourbaki, {V}ol.\ 9}, pages Exp. No. 297,
  287--307. Soc. Math. France, Paris, 1995.

\bibitem{HarariBook}
D.~Harari.
\newblock {\em Galois Cohomology and Class Field Theory}.
\newblock Springer International Publishing, 2020.

\bibitem{LECcompleto}
J.~S. Milne.
\newblock {\em \'{E}tale cohomology}, volume~33 of {\em Princeton Mathematical
  Series}.
\newblock Princeton University Press, Princeton, N.J., 1980.

\bibitem{GermanBook}
J.~Neukirch, A.~Schmidt, and K.~Wingberg.
\newblock {\em Cohomology of number fields}, volume 323 of {\em Grundlehren der
  Mathematischen Wissenschaften}.
\newblock Springer-Verlag, Berlin, second edition, 2008.

\bibitem{Sah1}
C.~H. Sah.
\newblock Cohomology of split group extensions.
\newblock {\em J. Algebra}, 29:255--302, 1974.

\bibitem{stacks-project}
T.~{Stacks Project Authors}.
\newblock \textit{Stacks Project}.
\newblock \url{https://stacks.math.columbia.edu}, 2020.

\bibitem{szamuely}
T.~Szamuely.
\newblock {\em Galois Groups and Fundamental Groups}.
\newblock Cambridge Studies in Advanced Mathematics. Cambridge University
  Press, 2009.

\bibitem{Shafarevich}
I.~R. \v{S}afarevi\v{c}.
\newblock On an existence theorem in the theory of algebraic numbers.
\newblock {\em Izv. Akad. Nauk SSSR. Ser. Mat.}, 18:327--334, 1954.

\bibitem{Weibel}
C.~A. Weibel.
\newblock {\em An Introduction to Homological Algebra}.
\newblock Cambridge University Press, apr 1994.

\end{thebibliography}
	
\end{document}